\documentclass[11pt]{amsart}
\usepackage{amsmath,amsthm,amsfonts,amssymb}
\usepackage{ifthen}

 \provideboolean{showlabels}
 \setboolean{showlabels}{false}
\newcommand{\mylabel}[1]{\ifthenelse{\boolean{showlabels}}{{\tt{[{#1}]}}\label
{#1}}{\label{#1}}}

\theoremstyle{plain}

\newtheorem{thm}{Theorem}
\newtheorem{lemma}{Lemma}
\newtheorem{proposition}{Proposition}

\theoremstyle{definition}

\newtheorem{defn}{Definition}

\numberwithin{equation}{section}

\newcommand{\Irr}{\operatorname{Irr}}

\newcommand{\F}{\mathbf{F}}

\newcommand{\C}{\mathbf{C}}
\newcommand{\p}{\mathcal{P}}

\newcommand{\Fst}{\F^{s \times t}}

\begin{document}

\title{Counting characters of small degree in upper unitriangular groups}
\author{Maria Loukaki}

\address{Department of Mathematics,
 University of Crete,
71409 Iraklio Crete,  
Greece}
\email{loukaki@gmail.com}

\begin{abstract}
Let $U_n$ denote the group of upper $n \times n$ unitriangular matrices over a fixed finite field $\F$ of order $q$. 
That is, $U_n$ consists of upper triangular $n \times n$ matrices having every diagonal entry equal to $1$.
It is known that the degrees of all irreducible complex characters of $U_n$ are powers of $q$.
 It was conjectured by Lehrer that the number of irreducible characters of $U_n$  
of degree  $q^e$ is an integer  polynomial in $q$ depending only on $e$ and $n$.
 We show that there exist recursive (for $n$) formulas that this number satisfies when $e$ is one of $1, 2$
 and $3$, and thus show that the conjecture is true in those cases.
\end{abstract}

\maketitle

\section{Introduction}
We fix a prime $p$. 
Let $q$ be a fixed power of $p$ and $\F_q = \F$ the  finite field of order $q$. We write $U_n(q)= U_n$ for 
 the group of upper triangular $n \times n$ matrices over $\F$, whose diagonal elements are 
  all  equal to $1$. We also write $GL_n(q)$ for the general linear  group of all $n \times n$  
invertible matrices over $\F$ and note that $U_n(q)$ is a $p$-Sylow subgroup of $GL_n(q)$. 
Furthermore, for every finite group $G$ and every integer $k$ we write 
$$
N_k(G)= |\{ \chi \in \Irr(G) \mid \chi(1) = k\}|,  
$$
for the number of irreducible characters of $G$ of degree $k$. 

In 1974 G. I. Lehrer, see \cite{leh}, conjectured two results. First, he  claimed that the
 degrees of the irreducible representations of $U_n$  are of type $q^e$ for some $e \in \{ 0, 1, \cdots, \mu(n)\}$, where 
$$
\mu(n)=\begin{cases}
m(m-1) \text{  if $n =2m $  and } \\
m^2  \text{ if $n= 2m+1$.} 
\end{cases}
$$ 
Next he conjectured that for any fixed $n$,  the number of irreducible characters of $U_n$ whose  degree  is  $q^e$, i.e., $N_{q^e}(U_n)$ in our notation,  is an integer polynomial in $q$ depending only on $e$.

As far as the first of his conjecture is concerned, 
it was shown  by M. Isaacs \cite{isa3},  that every irreducible character of $U_n$ has degree a power of $q$.
In addition, B. Huppert,   \cite{hup},  proved that the degrees of  the irreducible characters of $U_n$ 
is exactly the set $\{ q^e \mid 0 \leq e \leq \mu(n) \}$.

As for the second part of his conjecture, it still remains open apart for some specific values of $e$.

The case $e=0$ is well known and easy to compute,  that is,   $N_1(U_n(q))=N_1(U_n)= q^{n-1}$.
For greater values of $e$,  M. Marjoram  \cite{mar} provided some first  formulas. In particular, he proved that there exist formulas for   the number of irreducible characters having one of the 
 next two lowest degrees,  that  is  $N_{q}(U_n)$ and $N_{q^2}(U_n)$. 
Also in his unpublished thesis  \cite{mar2}, M. Marjoram 
 established formulas for  the three highest degrees  when $n=2m$ is even, that is $N_{q^{\mu(n)}}(U_{2m})$,
$N_{q^{\mu(n)-1}}(U_{2m})$ and $N_{q^{\mu(n)-2}}(U_{2m})$, as well as a formula for the 
number of irreducible characters of highest degree when $n $ is  odd, that is $N_{q^{\mu(n)}}(U_{2m+1})$.

In addition, I. M. Isaacs,  in his paper \cite{isa2},  using a different method, constructed specific polynomials for the  number of 
irreducible characters of $U_n(q)$  of degree $ q, q^{\mu(n)}$ and $q^{\mu(n)-1}$. He also suggested   a stronger form a Lehrer's conjecture (see Conjecture B in \cite{isa2}), that  the functions $N_{q^e}(U_n)$ are   polynomials in $q-1$  with nonnegative integer coefficients.

In this paper we use an elementary method to construct   recursive formulas that  the number of irreducible characters of  degree $q, q^2$ and  $q^3$ satisfy and thus we verify Lehrer's conjecture.
(Of course the cases $q$ and $q^2$ were already known  by Marjoram's formulas.)
In a forthcoming paper 
we prove analogous recursive formulas for  the degrees $q^{\mu(n)}, q^{\mu(n)-1}$ and $   q^{\mu(n)-2}$.

We follow the notation used in \cite{isa1}. In addition, 
for any matrix $X= (x_{i, j})  \in GL_n(q)$ we write $R_i(X)$   for its $i$-row  written as an $1 \times n$ matrix. We also write  $C_j(X)$ for its $j$-column 
written as an $n \times 1$ matrix.
Also if $A= (a_{i, j}) \in U_n $ then we say that its $i$-row is trivial if the only nonzero element in that row is the diagonal element $a_{i, i}=1$.
 Similarly, we say that the $j$-column of $A$ is trivial if every entry in the $j$-column of $A$ is $0$ except $a_{j,j}=1$.
We will often consider the additive abelian group  of the $st$-dimensional vector space $\F^{st}$ (of order $q^{st}$)
as  the additive group of  all $s \times t$ matrices over $\F$. When viewed as such we write it as $\F^{s \times t}$.

{\bf Acknowledgments}
Most of this work was done will I was visiting Georgia Institute of Technology. I would like to thank the Math 
department for its hospitality.   I'm  also grateful to the referee  whose suggestions improved dramatically the original paper. The short proof of Theorem 1 is his contribution. Furthermore, as he pointed out
the recursive formulas we prove below for $N_{q^3}(U_n)$  (see  \eqref{nq13}, \eqref{q^3.1}  and  \eqref{q^3.2})
imply that $N_{q^3}(U_n)$ is a polynomial function in $q-1$ with nonnegative integer coefficients. Hence 
 Isaacs conjecture holds for $e=3$. (Of course it holds for $e=1$ and $e=2$).

\section{Orbits of unitriangular actions on $\Fst$ }\mylabel{s1}

The aim of this section is to compute the orbits of a specific action of $H=U_s \times U_t$ on $\Fst$.

\begin{defn}\mylabel{df1} 
Let $T$ be an  $s \times t$ matrix over $\F$. 
We call  $T$   {\bf quasimonomial } if it has  at most 
one non-zero entry  in every column and  row.
\end{defn}
We write $ E_{i,j} $ for the matrix that has $1$ in it's $(i,j)$-entry and $0$ everywhere else. Clearly $E_{i,j}$ is quasimonomial. Furthermore, every nonzero quasimonomial  matrix $T$ can be written as  
\begin{equation}\mylabel{eq1}
T= f_{1}E_{i_1, j_1} +f_{2}E_{i_2, j_2}+ \cdots +f_{k}E_{i_k, j_k}
\end{equation}
 with $j_1 < j_2 < \cdots < j_k$, all $i_1, \dots , i_k$ distinct, 
  and  $f_1, \dots, f_k$ non-zero  elements in $\F$. We call \ref{eq1} the  {\bf standard form} of
the non-zero   $T$ and we  say that $k$  is the {\bf  length of} $T$. 

\begin{thm}\mylabel{thm1}
 Assume that the group $H= U_s \times U_t $ acts on $\Fst$
in the following way
$$
X^{(A , B)} = A^{-1} X  B, 
$$
for all $ X \in \Fst, A \in U_s$ and $ B \in U_t$.
Then the set of distinct quasimonomial matrices  in $\Fst$ forms a complete set of orbit representatives of 
the action of $H$ on $\Fst$. 
\end{thm}

\begin{proof}
 Let $X \in \Fst$. We show that by performing admissible transformations we can get a quasimonomial 
matrix. By an admissible transformation we mean adding to a row (respectively a column) 
a multiple of a subsequent row (resp. a previous column). By induction we can suppose that the $(s-1) \times t$ submatrix  of $X$ formed by all rows except the first one is quasimonomial.
Let $x_{i_1, j_1}, \ldots, x_{i_l, j_l}$, $2\leq i_1<  \ldots <i_l$, be the non-zero  elements in this submatrix.
Then  we can suppose that $x_{1, j_1}= \ldots = x_{1, j_l}=0$. If the  rest of elements in the first row are now zero we are done. Otherwise let $x_{1,j}$ be the first non-zero element in the first row. Then, except for 
$x_{1,j}$,  the $j$th column is zero and, by performing admissible column transformations, we can have that $x_{1,j}$ is the unique non-zero element in the first row, and thus obtain a quasimonomial matrix. 

To prove uniqueness, we argue again by induction on $s$ and $t$. If $X, Y$ are quasimonomial 
matrices in the same orbit, we can suppose that the last $s-1$ rows and the first $t-1$ columns of $X$ and $Y$ are the same. We only need to show that $x_{1,t}=y_{1,t}$. If some element in the first row or in the last column different from the $(1,t)$-entry is non-zero, then $x_{1,t}=y_{1,t}=0$ and $X=Y$. Otherwise comparing 
the $(1,t)$-entry in $XB=AY$ we get  $x_{1,t}=y_{1,t}$ and $X=Y$.
 
\end{proof}

When a first version of this paper appeared,  Vera-L\'{o}pez, Arregi and  Ormaetxea  
told me (I thank them for that) about a more general result concerning conjugacy classes in unitriangular groups (see \cite{Va1}, \cite{Va2}, \cite{Va3}), whose special case is  Theorem \ref{thm1}.

\section{Irreducible characters in $U_n$}
In this section we will show how  Section \ref{s1} is connected to Lehrer's conjecture. 
We  follow  Marjoram's  approach on the problem, and Proposition \ref{pr1} below follows
from his paper \cite{mar}.

For a fixed but arbitrary integer $n$ we consider  the upper unitriangular group $U_n$ over $\F_q$, and  its  two subgroups 
$M_{n,t}$  and $H_{n,t}$ defined in the following way.
If $1 \leq t \leq n$ and $s = n-t$ then  
$$
M_{n,t}= \{ \left(\begin{matrix}
I_{t} &  C \\
  0     &  I_s       \end{matrix}
\right): C \in \F ^{t \times s} \} = \{ X \in U_n \text{ with  }  x_{i,j} = 0  \text{ if  either \, $i< j \leq t $ 
or $t < i < j$}  \}
$$
 and 
$$
H_{n, t} = \{ \left(\begin{matrix}
B  &  0 \\
  0     &  A       \end{matrix}
\right): A \in U_s  \text{ and } B\in U_t \} = \{ X \in U_n \text{ with  }  x_{i,j} = 0  \text{ if   $i \leq t$
 and  $j > t$}  \}.
$$
It is easy to see that for all $t=1, \dots, n$,  the group  $M_{n, t}$ is an abelian normal subgroup  of $U_n$, 
isomorphic to $\F^{t \times s}$.  In addition, $H_{n,t}$  complements $M_{n,t}$ in $U_n$, and is isomorphic to $U_s \times U_t$.
We identify $H_{n,t}$ with $U_s \times U_t$ and we write its elements   as $(A, B)$ with $A \in U_s$ and $B \in U_t$.
We  also identify the
$M_{n,t}$ with $\F^{t\times s}$, and thus we write the elements of $M_{n,t}$ as $C \in \F^{t \times s}$.
Note that with these identifications,  the conjugation action of $H_{n,t}$ on $M_{n,t}$ in $U = M_{n,t} \rtimes H_{n,t}$  is given as 
\begin{equation}\mylabel{cong.action}
C^{(A,B)}= B^{-1}C A,
\end{equation}
for all $(A,B) \in H_{n,t}$ and $C \in M_{n,t}$. The product that appears at the right hand side of the equation 
above is the standard product of matrices.

Marjoram has given a nice characterization for the  abelian group $\Irr(M_{n,t})$, and the way $H_{n,t}$ acts on that group. 
We collect his results (Lemma 2 and 3 in \cite{mar})  in the next proposition. 
\begin{proposition}\mylabel{pr1}
Let $M_{n, t}$ and $ H_{n,t}$ be  defined as above, for fixed but arbitrary $n $ and $t$.
Then $\Irr(M_{n,t})$ is isomorphic to the  abelian additive group  $\Fst$ of all the $s \times t$ matrices over
 $\F = \F_q$. 
The isomorphism is given by
 $D \in \Fst \to \lambda_{D} \in \Irr(M_{n,t})$, 
where the map $\lambda_{D} : M_{n,t} \to \C$ is defined as  
$$
\lambda_{D}(C) = \omega^{T(tr(DC))},  \text{ for all } C \in M_{n,t}, 
$$
where  $\omega$ is a primitive $p$-root of unity, $T: \F_q \to \F_p$ is the usual trace map  from the extension 
field of $q$ elements $\F_q$ to the ground field $\F_p$ of $p$ elements and $tr(DC)$ denotes the trace of the square 
$s \times s$ matrix $DC$. Furthermore identifying  $H_{n, t} $ with $U_s \times U_t$, we get that the action of $H_{n,t}$ on $\Irr(M_{n,t})$ is given as 
$$
\lambda_{D}^{(A,B)}(C)= \lambda_{D}(C^{(A^{-1}, B^{-1})})=\lambda_{D}(BCA^{-1})=   \omega^{T(tr(DBCA^{-1}))}=
 \omega^{T(tr(A^{-1}DBC))}=\lambda_{A^{-1}DB}(C),
$$
for all $D \in \Fst, C \in \F^{t \times s}$ and all $(A, B) \in U_s \times U_t \cong H_{n,t}$.
Thus $ U_s \times U_t \cong H_{n,t}$ acts  on $\Irr(M_{n,t})\cong \Fst$ as 
\begin{equation*}
D^{(A,B)}=A^{-1}D B.
\end{equation*}
\end{proposition}

 What the above proposition says is that, 
identifying $\Irr(M_{n,t})$ with $\F^{s \times t}$ and $H_{n,t}$ with $U_s \times U_t$,  
 then   Theorem \ref{thm1} provides a complete set of orbit representatives of the action of
 $H_{n, t}$ on $\Irr(M_{n,t})$.
In particular,  
\begin{equation}\mylabel{orbits}
\Omega_{n,t}= \{ T \in \F ^{s \times t} \mid T  \text{ quasimonomial}\}.
\end{equation}
is such a set of representatives.

Now, let    $G$ be any finite group $N$ an  abelian normal subgroup of $G$ and $H$ a complement of $N$ in $G$, then it is easy to characterize the irreducible characters of $G$. In particular, if 
 $\lambda \in \Irr(N)$, and $G_{\lambda}$ is the stabilizer of $\lambda$ in $G$,  then  Gallagher's theorem 
 and  Clifford Theory  implies that 
$\lambda$ extends to $G_{\lambda}$ and a canonical extension $\lambda^e$ is given as $\lambda^e(hn) = \lambda(n)$, 
for all $h \in H_{\lambda}= G_{\lambda} \cap H$ and $n \in N$.  Every character $\Psi \in \Irr(H_{\lambda})$ 
defines a unique irreducible character $\Psi\cdot \lambda^e$ of $G_{\lambda}$ lying above $\lambda$ and inducing irreducibly on $G$. 
Distinct irreducible characters $\Psi \in \Irr(H_{\lambda})$ define distinct irreducible characters $(\Psi\cdot \lambda^e)^G$  of $G$. In addition, every $\chi \in \Irr(G)$ lies above some $\lambda \in \Irr(N)$ and thus $\chi = (\Psi\cdot \lambda^e)^G$, for some $\Psi \in \Irr(H_{\lambda})$. Note that
 $\chi(1)= \Psi(1) (|H|/|H_{\lambda}|)$.

The group $G$  acts on $\Irr(N)$ and divides its members into conjugacy classes. (Observe that the $G$-classes of $\Irr(N)$ 
are also   the $H$-conjugacy classes of $\Irr(N)$.)
Let  $\Omega \subseteq \Irr(N)$  consisting of one representative from   every $G$-conjugacy 
class of irreducible  characters of $N$. 
Then  $$
\Irr(G)= \cup_{\lambda \in \Omega} \{ (\Psi\lambda^e)^G \mid  \Psi \in \Irr(H_{\lambda}) \}. 
$$
Hence if $N_k(G)= |\{ \chi \in \Irr(G) \mid \chi(1)= k\}|$, for any finite group $G$, and any $k= 1,2 \cdots$, then 
\begin{equation*}\mylabel{eq2}
N_k(G)= \sum_{\lambda \in \Omega}N_{\frac{k}{|H:H_{\lambda}|}}(H_{\lambda})= \sum_{\lambda \in \Omega}N_{\frac{k}{|O_{\lambda}|}}(H_{\lambda})
\end{equation*}
where $O_{\lambda}$ is the $H$-orbit of $\lambda$ in $\Irr(N)$.
 
Applying the above argument to the groups $U_n= M_{n,t}\rtimes H_{n,t}$ for any arbitrary but fixed integer $n$ and any $t=1, \dots, n-1$, we conclude that 
\begin{equation}\mylabel{eq.nk}
N_{k}(U_n)= \sum_{T \in \Omega_{n,t}}N_{\frac{k}{|H_{n,t}:H_{n,t,T}|}}(H_{n,t,T})=\sum_{T \in \Omega_{n,t}}N_{\frac{k}{|O_T|}}(H_{n,t,T}) , 
\end{equation}
where  $\Omega_{n,t}$ is the set of quasimonomial matrices in $\Fst$,  $O_T$ is the  $H_{n,t}$-orbit
 of $T \in \Fst \cong \Irr(M_{n,t})$ and 
$H_{n,t,T}$ is the stabilizer of $T$ in  $ H_{n,t} \cong U_s \times U_t $.

{\bf Case 1}: $t=1$.  So $s=n-1$ and  the  groups $H_{n,1}$ and $M_{n,1}$ become  $U_{n-1} \times U_1 \cong U_{n-1}$ 
and $\F^{1 \times n-1}$  respectively. Furthermore, $\Irr(M_{n,1}) \cong \F^{n-1 \times 1}$ 
and 
$\Omega_{n,1}= \{ T \in \F^{n-1, 1} \mid T \text{ quasimonomial} \}$
 consists of the matrices 
$T_i = f E_{i, 1}$, for all $i=1, \dots, n-1$, and $f\ne 0 \in \F$, along with the zero matrix. 
So we get $q-1$ matrices of type $f E_{i,1}$. 
For any $n$ and any $i=1, \dots, n$ we define $P_{n,i}$ as 
$$
  P_{n,i}= \{ A \in U_n \mid C_i(A) \text{ is trivial}\}.
$$
Then it is easy to check that $H_{n,1,T_i}=P_{n-1,i}$ while $|O_{T_i}|= q^{i-1}$.
Thus in view of  equation \ref{eq.nk} we get 
\begin{equation}\mylabel{eq.s1}
N_{k}(U_n)= \sum_{T \in \Omega_{n,1}}N_{\frac{k}{|O_T|}}(H_{n,1,T})= (q-1)\sum_{i=1}^{n-1}N_{\frac{k}{q^{i-1}}}(P_{n-1,i})
 +N_k(U_{n-1}), 
\end{equation}
where the last summand is the contribution of the zero matrix whose orbit size is $1$ and the stabilizer group is $H_{n,1} \cong U_{n-1}$  itself.  For  $k=q^e$,  $e=0,1, \dots, \mu(n)$ the above equation, along with  
 the fact that $P_{n-1, 1}= U_{n-1}$,  implies
\begin{equation}\mylabel{eq.s2}
N_k(U_n)= qN_k(U_{n-1}) + (q-1)\sum_{i=2}^{n-1}N_{\frac{k}{q^{i-1}}}(P_{n-1,i}). 
\end{equation}
Observe  that for $k=1$ equation  \eqref{eq.s2} provides the well known formula
$N_1(U_n) = qN_1(U_{n-1})$, for all $n\geq 2$.

{\bf Case 2}: $t=2$  and thus  $s=n-2$. (Assume $n \geq 4$ for the rest of the section.)
 Now   the  groups $H_{n,2}$ and $M_{n,2}$ become  $U_{n-2} \times U_{2} \cong   U_{n-2} \times \F$ 
and $\F^{2 \times n-2}$  respectively. Furthermore, $\Irr(M_{n,2}) \cong \F^{n-2 \times 2}$ 
and 
$\Omega_{n,2}= \{ T \in \F^{n-2, 2} \mid T \text{ quasimonomial} \}$
 consists of   matrices whose length is either $1$ or $2$ along with the zero matrix. In particular,
 the non-zero 
matrices in $\Omega_{n,2}$ are of the following two types: \\
Those of length $1$, i.e.  $T_{i,j} = f E_{i, j}$,  $j=1, 2$ and $i=1, \dots, n-2$, while $f\ne 0 \in \F$. 
For any fixed $i$ and $j$ we  get   $q-1$ such.
If $j=1$ then $T_{i,1} = fE_{i,1}$, for $i=1, \dots, n-2$. In this case it is left to the reader to check that    $|O_{T_{i,1}}|=q^{i}$ while $
H_{n,2,T_{i,1}}=\{ (A, B) \mid  A \in U_{n-2}, \,  B \in U_{2}  \text{ with $C_{i}(A)$ and $ R_{1}(B)$   trivial } \}$. 
Thus $H_{n,2,T_{i,1}}\cong   P_{n-2, i}$.\\
If $j=2$ then $T_{i,2} = fE_{i,2}$, for some $i=1, \dots, n-2$. In this case $|O_{T_{i,2}}|=q^{i-1}$ while 
$H_{n,2,T_{i,2}}=\{ (A, B) \mid A \in U_{n-2}, \,   B \in U_{2} \text{ with $C_i(A)$ and $ R_{2}(B)$ being  trivial } \} \cong    P_{n-2, i} \times \F$.

The second type are those of length $2$, i.e.,  $T_{i_1, i_2} = f_1E_{i_1, 1} +f_2 E_{i_2, 2}$ for some $i_1\ne i_2$ and $f_1, f_2$ non-zero 
elements in $\F$. We get exactly $(q-1)^2$ such distinct quasimonomial characters.\\
One can easily check that if $i_1 > i_2$, then  $|O_{T_{i_1, i_2}}|=q^{i_1+i_2-1}$, while 
 the stabilizer of $T_{i_1, i_2}$ in $H_{n, 2}$ equals
$$
H_{n, 2, T_{i_1, i_2}}= \{ (A, 1) \mid A \in U_{n-2} \text{ with    $C_{i_1}(A)$ and $ C_{i_2}(A)$ being  trivial} \} \cong P_{n-2, i_1} \cap P_{n-2,i_2}.
 $$
On the other hand if  $i_1 < i_2$, then   $|O_{T_{i_1, i_2}}|=q^{i_1+i_2-2}$, while
the stabilizer $H_{n,2, T_{i_1,i_2}}$ of $T_{i_1, i_2}$ in $H_{n, 2} = U_{n-2} \times U_{2}$ consists of all matrices
 $(A, B) \in U_{n-2} \times U_{2}$ that satisfy $a_{i_1,i_2}= -f_1/f_2 \cdot b_{1, 2}$ while $C_{i_1}(A)$ is a trivial column and 
$a_{x,i_2}= 0$ for all $i_1 \ne x =1, \dots, i_2-1$. 
For $1 \leq i_1 < i_2 \leq n$ we define 
\begin{equation}\mylabel{defq}
Q_{n,i_1,i_2}= \{ A \in U_{n} \mid a_{y,i_1}=0= a_{x, i_2}, \text{ for all $ i_1 \ne x=1, \dots, i_2-1$ and  $ y= 1,  \dots, i_1-1$}  \}.
\end{equation}
Then it is easy to see that
$H_{n,2, T_{i_1,i_2}} \cong Q_{n-2, i_1, i_2}$.
Finally the zero matrix has orbit length $1$ and its stabilizer in $H_{n,2}$ is $H_{n,2} \cong  U_{n-2} \times \F$.
Collecting all the above  and applying  equation \ref{eq.nk} along with equation \eqref{eq.s2} 
and the fact that $N_k( M \times \F)=|\F| \cdot N_{k}(M)= q N_k(M)$  for  any group $M$,  we get  
\begin{multline}\mylabel{eq.s3}
N_k(U_n)= qN_k(U_{n-1})+N_{\frac{k}{q}}(U_{n-1})- N_{\frac{k}{q}}(U_{n-2})+\\
(q-1)^2\sum_{1 \leq i_2 < i_1 \leq n-2}N_{\frac{k}{q^{i_1+i_2-1}}}(P_{n-2,i_1} \cap P_{n-2, i_2})+ \\
(q-1)^2\sum_{1 \leq i_1 < i_2 \leq n-2}N_{\frac{k}{q^{i_1+i_2-2}}}(Q_{n-2,i_1,i_2}).
\end{multline}
for all $n \geq 4$ and all $k$.
Some of the summands  above  are easy to compute. First observe that $P_{n-2, i} \cap P_{n-2,n-2}\cong P_{n-3,i}$ for all $i=1, \dots, n-3$, and all $n \geq 5$.  Thus \eqref{eq.s2} implies  
$$(q-1)^2\sum_{i=1}^{n-3}N_{\frac{k}{q^{n-3+i}}}(P_{n-2,i} \cap P_{n-2, n-2})=
(q-1)[N_{\frac{k}{q^{n-2}}}(U_{n-2}) - N_{\frac{k}{q^{n-2}}}(U_{n-3})].
$$
 In addition, $P_{n-2, i } \cap P_{n-2, 1} = P_{n-2, i}$, for all $i=2, \dots, n-3$ and all 
$n \geq 5$.
Hence 
$$(q-1)^2\sum_{i=2}^{n-3}N_{\frac{k}{q^{i}}}(P_{n-2,i} \cap P_{n-2, 1})=
(q-1)[N_{\frac{k}{q}}(U_{n-1}) - qN_{\frac{k}{q}}(U_{n-2})- (q-1)N_{\frac{k}{q^{n-2}}}(U_{n-3})].
$$
Furthermore, for all $i=1, \dots, n-3$  and all $ n \geq 5$, 
 the group $Q_{n-2, i, n-2}$ is isomorphic to $P_{n-3, i} \times \F$. This along with \eqref{eq.s2}  implies
\begin{equation}\mylabel{eq.s3.3}
(q-1)^2\sum_{i=1}^{n-3}N_{\frac{k}{q^{n-4+i}}}(Q_{n-2,i,n-2})=
q(q-1)[N_{\frac{k}{q^{n-3}}}(U_{n-2}) - N_{\frac{k}{q^{n-3}}}(U_{n-3})].
\end{equation}
We finally observe that $Q_{n-2, 1, 2} = U_{n-2}$.
Replacing all the above in  equation \eqref{eq.s3}  we get 
\begin{multline}\mylabel{eq.s4}
N_k(U_n)= qN_k(U_{n-1})+qN_{\frac{k}{q}}(U_{n-1})- qN_{\frac{k}{q}}(U_{n-2}) +\\
(q-1)[N_{\frac{k}{q^{n-2}}}(U_{n-2})-qN_{\frac{k}{q^{n-2}}}(U_{n-3})
+qN_{\frac{k}{q^{n-3}}}(U_{n-2})-  qN_{\frac{k}{q^{n-3}}}(U_{n-3})]+\\
(q-1)^2\sum_{2 \leq i_2 < i_1 \leq n-3}N_{\frac{k}{q^{i_1+i_2-1}}}(P_{n-2,i_1} \cap P_{n-2, i_2})+ \\
(q-1)^2\sum_{1 \leq i_1 < i_2 \leq n-3  \text{ and } (i_1, i_2) \ne (1,2)}
N_{\frac{k}{q^{i_1+i_2-2}}}(Q_{n-2,i_1,i_2}), 
\end{multline}
for all $n \geq 5$ and all $k$.
Note that in the equation above, the sum $i_1+i_2$ is greater or equal to 5 when $i_2 < i_1$, while 
$i_1+i_2\geq 4$ when $i_1 < i_2 $.

\section{ Linear characters of  $P_{n, i}$ and $Q_{n,i,j}$ }\mylabel{section.p}
The aim of this section is to compute the number of linear characters of $P_{n, i}$ and $Q_{n,i,j}$.
These groups are examples of \textit{ pattern groups }, a term introduced by  M. Isaacs.
We give here the basic definitions and properties we need, for more details the reader could see \cite{isa2}. 

Let   $\p$  be a subset of the set of pairs $\{ (i, j) \mid  1 \leq i < j \leq n \}$. 
 $\p$ is called a \textit{closed pattern} if it has the property  that $(i, k) \in \p$ 
whenever $(i, j) , (j, k) \in \p$, for some $j\in \{i+1, \ldots , k-1\}$.
The set of unitriangular matrices $X \in U_n$ whose nonzero entries are restricted to lie at positions in the pattern $\p$  is a subgroup of $U_n$ called a \textit{pattern group}. If $G$ is a pattern group corresponding to the closed pattern $\p$ with $|\p|=k$, then $G$ is generated by the matrices $I_n +aE_{i,j}$, $(i,j) \in \p$, $a \in \F^*$ and $|G|=|\F|^k$.

Direct computations show that $[I_n +aE_{i,j}, I_n+bE_{l,k}]=I_n + abE_{i,k}$ if $j=l$ and $I_n$  otherwise.
A pair $(i,j) \in \p$ is called \textit{minimal} if it is not possible  to find numbers $j_1< j_2< \ldots < j_l, l \geq 1$, such that $(i, j_1), (j_1, j_2), \dots, (j_l, k) \in \p$. Then $G'$ is the pattern group
 associated  to $\p_0 = \{ (i, k) \in \p \mid (i, k) \text{ is not minimal} \}$. 
 Thus $|G:G'|=q^t$, where $t$ is the number of minimal pairs in $\p$ (see Theorem 2.1 in \cite{isa2}). 

For the group $P_{n, i}$,  $n \geq 3, i \leq n-1$   observe that there are $n-1$  minimal pairs:
$(k, k+1), k \ne i-1,  1\leq k \leq n-1$ and  $(i-1, i+1)$. Therefore 
\begin{equation}\mylabel{s4.3}
 N_1(P_{n, i}) = q^{n-1}.
\end{equation}

For the group $Q_{n, i, i+1}$ with  $2 \leq i \leq n-2$ there are $n-1$ minimal pairs:
$(k, k+1), k \ne i-1, 1 \leq k \leq n-1$  and   $(i-1, i+2)$.
For the group $Q_{n, i, j}$ with  $1 < i < j-1 \leq n-1$, there are $n$ minimal pairs:
$(k, k+1), k \ne i-1, j-1, 1 \leq k \leq n-1$  and  $(i-1, i+1), (i,j),  (j-1, j+1)$. Therefore 
\begin{equation}\mylabel{hg}
N_1(Q_{n,i,j})=  \begin{cases}   
q^{n-1} \text{ if $i=j-1$ } \\
q^{n} \text{ if $i <j-1$ }.\\
\end{cases}
\end{equation}

\section{ Computing $N_q(P_{n,2})$ }
With the aid of  $N_1(P_{n, i})$ and $N_1(Q_{n,i, j})$ we  
 give the recursive formulas for $N_k(U_n)$ when $k=q$ and $k= q^2$ and compute $N_q(P_{n,2})$.
For $k=q$ and $n \geq 5$,  equation \eqref{eq.s2}, implies
\begin{equation}\mylabel{eq.k=q}
N_q(U_n)= qN_q(U_{n-1}) +(q-1)N_1(P_{n-1,2})= qN_q(U_{n-1})+ q^{n-2}(q-1).
\end{equation}
It is straight forward to see that  $N_q(U_3)= q-1$ while $N_q(U_4)= q(q-1)(q+1)$.  So the formula 
$N_q(U_n)=q^{n-3}(q-1)((n-3)q+1)$  for $N_q(U_n)$ obtained by both Marjoram \cite{mar} and Isaacs \cite{isa2}
satisfies \eqref{eq.k=q}.

For $k=q^2$ and $n=5$ equation \eqref{eq.s4} implies $N_{q^2}(U_5)= q(q-1)(2q^2+q-1)$.
In addition, for all $n \geq 6$ we have
 \begin{multline}\mylabel{k=q^2.rec}
N_{q^2}(U_n)=qN_{q^2}(U_{n-1}) +qN_q(U_{n-1})- qN_q(U_{n-2})+(q-1)^2N_1(Q_{n-2, 1, 3})=\\\ qN_{q^2}(U_{n-1})+q^{n-4}(q-1)[q^3+(n-5)q^2-(n-6)q -1].
\end{multline}
It is straight forward to check that the above recursive formula is satisfied by 
\begin{equation}\mylabel{k=q^2}
N_{q^2}(U_n)= q^{n-4}(q-1)\{(n-5)q^3+(\frac{(n-5)(n-4)}{2}+2)q^2+
[1-\frac{(n-6)(n-5)}{2}]q -n+4\}.
\end{equation}
On the other hand equation  \eqref{eq.s2} for $k=q^2$ and $n \geq 5$,  implies
\begin{equation}\mylabel{e10}
N_{q^2}(U_n)= qN_{q^2}(U_{n-1})+(q-1)[q^{n-2}+ N_{q}(P_{n-1,2})].
\end{equation} 
Combining the above with  \eqref{k=q^2.rec} we get  
\begin{equation}\mylabel{nqp42}
 N_q(P_{4, 2})= \frac{1}{q-1}(N_{q^2}(U_5)-qN_{q^2}(U_4))- q^3= q(q^2-1),
\end{equation}
while for $n \geq 6$ 
\begin{equation}\mylabel{nqpn2}
N_q(P_{n-1, 2}) = q^{n-4}(q-1)[q^2+ (n-5)q +1].
\end{equation}

\section{The group $Q_{n, 1, 3}$.}
The aim in this section is to compute $N_q(Q_{n, 1, 3})$, for all $n \geq 4$. We  point out that we are not able  to compute the number of irreducible characters of degree $q$ for every group $Q_{n, i, j}$ where $i,$ and $j$ are arbitrary. But we can do it for the group $Q_{n, 1, 3}$, and this is enough for the computation of $N_{q^3}(U_n)$.

Assume that $n \geq 4$. Note that, according  to its  definition, $Q_{n, 1, 3}$ consists of all $n \times n$ unitriangular matrices whose $(2,3)$-entry
is zero. We write $Q_{n, 1, 3}$  as a semidirect product using the following groups.
Let $M$ be the subgroup  of $Q_{n, 1, 3}$ consisting of matrices 
all of whose non-diagonal elements are zero except for  the first row. Assume further that 
$H$ is  the subgroup  of $Q_{n, 1, 3}$ consisting of matrices whose non-diagonal entries in the first  
row are zero. Then it is clear that $M$ is an  abelian normal subgroup of $Q_{n, 1, 3}$
isomorphic to $\F^{1 \times n-1 } \cong \F^{n-1}$. Observe that 
 $H$  is isomorphic to  $ P_{n-1, 2}$. Furthermore,  $Q_{n,1,3}= M \rtimes H$ and  the conjugation
 action of $H$ on $M$ is given as 
$$
\left(\begin{matrix}
1 &0   \\
0  &X^{-1} 
      \end{matrix}
\right)
\left(\begin{matrix}
1 &C \\
0 &I_{n-1}
      \end{matrix}
\right)
\left(\begin{matrix}
1 &0 \\
0 &X
      \end{matrix}
\right)
=
\left(\begin{matrix}
1 &CX \\
0 &I_{n-1}
      \end{matrix}
\right), 
$$
where $X \in P_{n-1, 2}$ and $C \in \F^{1 \times n-1}$.

Now we apply  Proposition \ref{pr1} to the group $M  = \F^{1 \times n-1 }$ (with  $M$ in the place of $M_{n, t}$ for $t=1$).  
So the group of irreducible characters  $\Irr(M)$ of $M$ is isomorphic to the abelian additive group $\F^{n-1\times 1}$. 
Thus we regard the irreducible characters of $M$ as column  vectors  over $\F$, and  for every $\chi \in \Irr(M)$ we write
$\chi = (\chi_1, \ldots ,  \chi_{n-1})^t $  with $\chi_i \in \F$. 
Under the isomorphism between $\Irr(M)$ and $\F^{n-1 \times 1}$ the action of an element $\left(\begin{matrix}
1 &0 \\
0 &X
      \end{matrix}
\right)$  becomes multiplication on the left by $X^{-1}$.
It is straightforward to see that the $H$-invariant irreducible characters of $M$ are 
those with $\chi_3=\chi_4= \ldots = \chi_{n-1}=0$, and thus they look like 
$(\chi_1, \chi_2, 0, \ldots, 0)^t$ with $\chi_1, \chi_2 \in \F$. Hence we get $q^2$ such irreducible characters.

Furthermore, if $\chi= (
\chi_1, \ldots, \chi_{n-1})^t
\neq (0, \ldots, 0)^t$ 
is any character  in $\Irr(M)$,  and  $k$ is the biggest index with $\chi_k \ne 0$, then if $k=1, 2$ the character  $\chi $
 is $H$-invariant , while if $k \geq 3$ the $H$-orbit of $\chi$ contains all the characters of type
$(f_1, \ldots, f_{k-1}, \chi_k, 0,  \ldots,  0)^t$, 
where $f_i \in \F$ are arbitrary. 
Hence we get orbits of length $q^{k-1}$.
Therefore for any  $\chi \in \Irr(M)$ either $\chi$ is $H$-invariant or its stabilizer $H_{\chi}$  in $H$ has index at least $q^2$. 
That is,  there are no irreducible characters in $\Irr(M)$ whose stabilizer in $H$ has index $q$ in $H$. 

Now we follow the argument after equation \eqref{orbits}, for the group $Q_{n, 1, 3} = M \rtimes H$, to get  
 $N_q(Q_{n, 1, 3})= q^2 \cdot N_q(P_{n-1, 2})$, 
for all $n \geq 4$. 
If $n =4$ then  $P_{3, 2} \cong \F^2$  and thus 
\begin{equation}\mylabel{nq413}
 N_q(P_{3, 2})= N_q(Q_{4, 1, 3})= 0.
\end{equation}
If $n =5$ then  in view of \eqref{nqp42}  we get 
\begin{equation}\mylabel{nq513}
N_q(Q_{5,1,3})= q^3(q^2-1)
\end{equation}
In addition, for all $n \geq 6$, we use \eqref{nqpn2} to get  
\begin{equation}\mylabel{nq13}
N_q(Q_{n, 1, 3})= q^{n-2}(q-1)[q^2+(n-5)q +1].
\end{equation}

\section{Computing $N_{q^3}(U_n)$.}
For $n=5$,  equation \eqref{eq.s4} implies that $N_{q^3}(U_5)=q(q-1)(2q-1)$.
Furthermore, when $k=q^3$ and $n=6$ equation \eqref{eq.s4} along with \eqref{nq413} and \eqref{hg} implies 
$$
N_{q^3}(U_6)=q^2(q-1)(4q^2+q-3).
$$
For the case $n =7$ we similarly get 
\begin{equation}\mylabel{n=7q3}
N_{q^3}(U_7)=q^2(q-1)[3q^4+6q^3-2q^2-5q +1].
\end{equation}

In general, for all $n \geq 8$  equation \eqref{eq.s4} implies 
\begin{multline}\mylabel{q^3.1}
N_{q^3}(U_n)=qN_{q^3}(U_{n-1})+qN_{q^2}(U_{n-1})-qN_{q^2}(U_{n-2})+\\
(q-1)^2[N_q(Q_{n-2,1,3})+N_1(Q_{n-2,2,3})+N_1(Q_{n-2,1,4})]. 
\end{multline}
According to \eqref{k=q^2}, for all $n \geq 8$ we get 
\begin{multline}\mylabel{q^3.2}
N_{q^2}(U_{n-1})- N_{q^2}(U_{n-2})= q^{n-6}(q-1)\{(n-6)q^4+[\frac{(n-6)(n-5)}{2}-(n-7)+2]q^3-\\
[(n-7)(n-6)+1]q^2+ [4-n+\frac{(n-8)(n-7)}{2}]q +n-6\}
\end{multline} 
Furthermore, using \eqref{hg} and   \eqref{nq13} along with \eqref{q^3.2} in equation \eqref{q^3.1}  and we get 
\begin{multline}\mylabel{q^3.3}
N_{q^3}(U_n)=qN_{q^3}(U_{n-1})+q^{n-5}(q-1)\{ q^5 +(2n-14) q^4 + [25-3n +\frac{(n-6)(n-5)}{2}]q^3 +\\
[n-11-(n-7)(n-6)]q^2+[5-n+\frac{(n-8)(n-7)}{2}]q +n-6\},
\end{multline}
for all $n \geq 8$. As we have already computed the formula for $N_{q^3}(U_7)$, we can easily check 
 that the following equation satisfies the recursive formula \eqref{q^3.3} for all $n \geq 8$.
\begin{equation}{\mylabel{q^3.4}}
N_{q^3}(U_n)=q^{n-5}(q-1)\{A_n q^5 +B_n q^4 + C_n q^3 +
D_n q^2+E_n q +F_n\},
\end{equation}
where
\begin{itemize}
 \item $A_n=n-7$\\
\item $B_n=3+(n-7)(n-6)$\\
\item $C_n=40(n-7) -\frac{17}{4}(n+8)(n-7)+\frac{1}{12}n(n+1)(2n+1)-64$\\
\item $D_n= (n-7)(7n+3) -\frac{1}{6} n(n+1)(2n+1)+138$\\
\item $E_n=(n-7)(-\frac{17}{4}n-1)+\frac{1}{12} n(n+1)(2n+1)-75$ \\
\item $F_n= 1+\frac{(n-7)(n-4)}{2}$.
\end{itemize}
It is clear that the polynomials  $A_n, B_n$ and $F_n$ in $n$ are integer valued for every $n$.
To show that the same holds for the polynomials $C_n, D_n$ and $E_n$ we make use of the following lemma.
\begin{lemma}\mylabel{lem.f}
 Let $P(n)$ be a polynomial in $n$ of degree $m$  with rational coefficients.
If $P(n)$  is an integer for $m+1$ consecutive integers, then the polynomial is integer valued. 
\end{lemma}
\begin{proof}
We will use induction on the degree $m$ of $P(n)$.
It is clear that for $m=1$ holds.

Assume it holds for all polynomials of degree less that $m$,  we will show that it also holds 
 for those of degree $m$. 
The polynomial $Q(n):= P(n+1) -P(n)$  has degree smaller than $m$. In addition, if $P$
has integer values for $m+1$  consecutive  integers $k, k+1, \dots, k+m$, then $Q(n)$ 
is  integer valued for  the $m$ consecutive integers $k, k+1, \dots, k+m-1$. 
Therefore the inductive hypothesis implies that $Q(n)$ is integer valued for every $n$. 
This along with the fact that $P(n)$ is integer valued for $n=k+m$, implies that $P(n)$ is  an integer for every $n$. 
\end{proof}

Now,  it is straight forward to check 
that $C_7, C_8, C_9$ and $C_{10}$ are integers. Hence the above lemma implies that 
 $C_n$ is an integer valued polynomial. Similarly   we show that $D_n$ and $E_n$ are integer valued. 
Hence $N_{q^3}(U_n)$ is a polynomial in $q$ with integer coefficients.

\end{document}